\documentclass[12pt]{amsart}
\usepackage{color,graphicx,array, amssymb, amscd,slashed,bm}
\usepackage{hyperref}
\newtheorem{Theorem}{Theorem}[section]
\newtheorem{Lemma}[Theorem]{Lemma}
\newtheorem{Proposition}[Theorem]{Proposition}
\newtheorem{Corollary}[Theorem]{Corollary}
\theoremstyle{definition}
\newtheorem{Definition}{Definition}
\theoremstyle{remark}

\newtheorem{Remark}[Theorem]{Remark} 

\numberwithin{equation}{section}
\renewcommand{\k}{\kappa}

\newcommand{\R}{\mathbb R}

\newcommand{\C}{\mathbb C}
\newcommand{\D}{\mathbb D}

\newcommand{\SOn}{{\rm SO}^+ (1, n+1)}

\newcommand{\SOo}{{\rm SO}^+ (1,1)}
\newcommand{\SOf}{{\rm SO}^+ (1,3)}
\newcommand{\SOnf}{{\rm SO} (n-2)}

\newcommand{\Gr}{\operatorname{Gr}_{1,1}}
\newcommand{\Grf}{\operatorname{Gr}_{3,1}}

\newcommand{\la}{\langle}
\newcommand{\ra}{\rangle}
\newcommand{\SSS}{\mathbb S}
\newcommand{\LL}{\mathbb L}
\newcommand{\G}{V}
\newcommand{\Z}{Z}

\newcommand{\real}{\mathbb R}

\newcommand{\di}{\operatorname{diag}}

\renewcommand{\Re}{\operatorname {Re}}
\renewcommand{\Im}{\operatorname {Im}}

\newcommand{\hy}{\widehat{Y}}

\newcommand{\red}[1]{{\leavevmode\color{red}{#1}}} 
\newcommand{\blue}[1]{{\leavevmode\color{blue}{#1}}}
\renewcommand{\red}[1]{#1} 
\renewcommand{\blue}[1]{#1} 
\usepackage{amsmath}	
\begin{document}

\title[Gauss maps of M\"obius surfaces in $\SSS^n$]{Gauss maps of M\"obius surfaces in the $n$-dimensional sphere}
\author{David Brander}
\address{Department of Applied Mathematics and Computer Science
Technical University of Denmark
2800 Lyngby (Copenhagen)
Denmark}
\email{dbra@dtu.dk}
 \author[S.-P.~Kobayashi]{Shimpei Kobayashi}
 \address{Department of Mathematics, Hokkaido University, 
 Sapporo, 060-0810, Japan}
 \email{shimpei@math.sci.hokudai.ac.jp}
\author{Peng Wang}
\address{School of Mathematics and Statistics, Key Laboratory of Analytical Mathematics and Applications (Ministry of Education), Fujian Key Laboratory of Analytical Mathematics and Applications (FJKLAMA), Fujian Normal University, Fuzhou 350117, P. R. China} \email{netwangpeng@163.com, pengwang@fjnu.edu.cn}

 \thanks{The first named author is partially supported by Independent Research Fund Denmark, grant 9040-00196B.
 The second named author is partially supported by JSPS 
 KAKENHI Grant Number JP18K03265, JP22K03304.
 The third named author is  supported by the Project 12371052 of NSFC}
 \date{\today}
\pagestyle{plain}
\begin{abstract}
\red{In this note we discuss Gauss maps for M\"obius surfaces in the
$n$-sphere, and their applications in the study of Willmore surfaces.
One such ``Gauss map'', naturally associated to a Willmore surface that has a dual Willmore surface, is the Lorentzian $2$-plane bundle given by 
a lift of the suface and its dual. More generally, we define the concept of a Lorentzian $2$-plane lift for an arbitrary M\"obius surface, and
show that the conformal harmonicity of this lift is equivalent to the Willmore condition for the surface. This clarifies some previous work of 
F. H\'elein, Q. Xia-Y Shen, X. Ma and others, and, for instance, allows for the treatment of the Bj\"orling problem for Willmore surfaces in the presence of umbilics.}
\end{abstract}
\maketitle
\section*{Introduction}
Gauss maps of submanifolds have been extensively used in the study of special submanifolds, including minimal submanifolds. To study surfaces in an $n$-dimensional M\"obius space, the \textit{conformal Gauss map}, which represents the central sphere congruence, proves particularly useful \cite{Bl}. 
It is known that a surface is Willmore if and only if its conformal Gauss map is conformally harmonic \cite{Bryant1984}, \cite{Ejiri1988}.

The conformal Gauss map can be regarded as a map into the  Grassmannian symmetric space of $4$-dimensional Lorentzian subspaces in the $(n+2)$-dimensional Minkowski space. This formulation enables an application of integrable systems methods. From the harmonicity of the conformal Gauss map, one can derive a family of flat connections, leading to a family of moving frames, i.e., a map into the loop group of the indefinite orthogonal group $\SOn$ as in \eqref{eq:loopgroup}. This perspective was initiated in \cite{Helein} and further developed in \cite{DoWaGeneric}.

While the conformal Gauss map uniquely characterizes a M\"obius surface and its harmonicity defines a Willmore surface, it has some limitations. For example, defining the conformal Gauss map requires information about a $4$-dimensional Lorentzian subspace determined by second-order contact. Additionally, at umbilic points, the conformal Gauss map degenerates, complicating the analysis. To address these issues, a more geometrically intuitive Gauss map has been proposed. This alternative map targets the Grassmannian of $2$-dimensional Lorentzian planes in the $4$-dimensional Lorentzian subspace, \cite{Helein}, \cite{Ma2006}.

This Gauss map can be easily constructed for special Willmore surfaces (the so-called S-Willmore surfaces, \cite{Ejiri1988} and
Section \ref{sbsc:SWill} below). It is defined by Lorentzian planes and denoted by $Z = [Y \wedge \widehat{Y}]$, where $Y$ is a lift of a Willmore surface, and $\widehat{Y}$ is its dual Willmore surface, and $Z$ is a conformal harmonic map. S. Ma extended this approach to a broader class of Willmore surfaces, including S-Willmore surfaces, and introduced the concept of the ``adjoint surface'', where $\widehat{Y}$ is another Willmore surface, not necessarily dual to $Y$ (\cite{Ma2006}). 
Moreover, this idea was thoroughly explored in \cite{Helein}, where a family of such Gauss maps was considered. F. Hél\'ein characterized Willmore surfaces under the condition of ``roughly harmonicity'' and applied integrable systems methods. 

 In this \red{note}, we define a generalization of this Gauss map for all surfaces in M\"obius space, referred to as an \textit{$\mathbb{L}$-lift} (Definition \ref{def:llift}). While the choice of an $\mathbb{L}$-lift is 
 not unique, this flexibility allows for a special selection such
 that an $\LL$-lift is conformal. Note that the conformality is characterized by the Riccati equation; thus, such an $\LL$-lift exists locally but not globally, in general. (However, there exist many globally defined conformal $\LL$-lifts, see Remark \ref{rmk:global}). Then, using a conformal $\mathbb{L}$-lift, we establish the equivalence between the harmonicity of the $\mathbb{L}$-lift and the Willmore property of the surface. We finally provide a characterization of Willmore surfaces through integrable systems methods (Theorem \ref{thm:RuhVilms}).

In Section \ref{sc:Harmoniclift}, we focus on S-Willmore surfaces and the adjoint transformation of a Willmore surface, analyzing them via an $\mathbb{L}$-lift with a Riccati equation. Furthermore, we investigate the properties of the moving frame associated with these structures (Proposition \ref{pp:various}).

The characterization of Willmore surfaces via the $\mathbb{L}$-lift presented in this short paper forms the basis for addressing the Bj\"orling problem. (The problem involves reconstructing surfaces from boundary curves and the Gauss map data along those curves). Our approach will be further explored in a forthcoming paper \cite{BKW}, focusing on minimal surfaces in $\mathbb{S}^n$ and $\mathbb{H}^n$.

{\bf Acknowledgments:}
\red{The third named author} is thankful to Prof. Xiang Ma and Prof. Zhenxiao Xie for valuable discussions.

\section{Preliminaries}
\subsection{M\"obius surfaces in $\mathbb S^n$ and the conformal Gauss maps}\label{sbsc:conformal}
 The following M\"obius equivalence is known \cite{Bl}:
$\mathbb S^n \cong \mathbb P (\mathcal C_+^{n+1})$ by $x \leftrightarrow [(1, x)]$,
 where 
 \[\mathcal C_+^{n+1}= \left\{ v \in \mathbb R^{n+1, 1} \mid \langle v, v\rangle= -v_0^2 +
 \sum_{j=1}^{n+1} v_j^2=0, \, v_0>0 \right\} \subset \mathbb R^{n+1, 1}
 \]
 is the \textit{forward light 
 cone} in the $n+2$-dimensional Minkowski space $\mathbb R^{n+1, 1}$.
  We define a \textit{M\"obius surface} to be an immersion $y: M \to  \mathbb S^n \cong \mathbb P (\mathcal C_+^{n+1})$, where the Riemann surface $M$ has the conformal structure 
 induced by any  M\"obius equivalent choice of Riemannian metric on $\SSS^n$.

 Let $M$ be a Riemann surface with conformal coordinate $z$ 
 and $y: M \to \mathbb S^n$ be a conformal immersion. 
 Let $Y: \mathbb D \subset M \to \mathcal C^{n+1}_+$ be a local lift of $y$, that is, 
 $[1,y]= [Y]$ holds. It is known that 
 the conformality of $y$ is equivalent to 
 $\langle Y_z, Y_z \rangle  =0$ and $\langle Y_{z}, Y_{\bar z} \rangle >0$, 
 where the subscript $z$ (resp. $\bar z$) denotes the partial derivative with respect to $z$ (resp. $\bar z$)
 and $\la\,, \,\ra$ denotes the complex bilinear extension of the inner product of $\R^{n+1, 1}$.
 Let us decompose the trivial bundle 
 $M \times \mathbb R^{n+1,1}$ as $V \oplus V^{\perp}$ with 
\begin{equation}\label{eq:V}
 V = \operatorname{span}\left\{ Y, \, \Re Y_z,\, \Im Y_z,\, Y_{z \bar z}  \right\},
\end{equation}
 and $V^{\perp}$ denotes the orthogonal complement of $V$ in $\mathbb R^{n+1, 1}$.
 Clearly $V$ is independent of the choice of a lift $Y$ 
 and the coordinate $z$, thus $V$ is a Lorentzian rank four sub-bundle.
 Let us denote $V_{\mathbb C}$ (resp. $V_{\mathbb C}^{\perp}$) 
 as the complexification of $V$ (resp. $V^{\perp}$). 
 Thus we reach   
 the following definition:
 \begin{Definition}\label{def:conGauss}
 For a M\"obius surface $y : M \to \mathbb S^n$, let $V$ be the sub-bundle in \eqref{eq:V} 
 and denote an evaluation at $p \in M$ by $V_p$.
 Then the map 
 \[
 \G: M \to \Grf = \SOn /( \SOf \times \SOnf), 
  \]
 defined by $\G(p)=[V_p]$  
 will be called the \textit{conformal Gauss map} 
 of $y$. Here $\SOnf$ and  $\SOn$ denote
 degree $n$ orthogonal group and 
 the identity component of degree 
 $n+2$ indefinite orthogonal group, respectively.  
\end{Definition}
\begin{Remark}
 It is well known that the conformal Gauss map $\G$ is 
 obtained by the mean curvature $2$-spheres of a M\"obius surface
 in $\mathbb S^n$, the central sphere congruence, through the identification $\SSS^n \cong  P (\mathcal C_+^{n+1})$, and it is 
 conformal for any M\"obius surface.
\end{Remark}
 \subsection{The canonical lift, the Willmore energy and Willmore surfaces}\label{sbsc:canonical}
 The \textit{canonical lift} with respect to $z$ is defined by a lift $Y$ such that 
\begin{equation}\label{eq:canonical}
 |d Y|^2 = |dz|^2
\end{equation}
 holds, that is, $\langle Y_{z}, Y_{\bar z}\rangle=1/2$ holds.
 For the canonical lift $Y$, it is easy to see that 
 $Y_{zz}$ is orthogonal to $Y_z, Y_{\bar z}$ and $Y_{z \bar z}$, and there 
 exists a complex function $s : \mathbb D \to \C$ and a section $\k \in \Gamma (V_{\mathbb C}^{\perp})$
 such that 
\begin{equation}\label{eq:Hill}
 Y_{z z} +\frac12{s} Y =\k
\end{equation}
 holds, that is, $Y$ satisfies the \textit{inhomogeneous Hill's equation}. 
 The function $s$ defined above
  will be called the \textit{Schwarzian derivative} of 
 the immersion $y$ with respect to the coordinate $z$ and it will be denoted by $S_z (y ) := s$,
 and the section $\k$ will be called the \textit{Hopf differential}
 of the immersion $y$.
 The reasons why we respectively call $s$ and $\k$ the Schwarzian derivative  and the Hopf 
 differential will be given below. (See Formula (23) and (24) of \cite{BPP}).
\begin{Proposition}[\cite{BPP}]\label{Pp:invariants}
 The Schwarzian derivative $s$ and the Hopf differential $\k$ 
 of an immersion $y$ are  invariant under M\"obius transformations. 
 Moreover, the quadratic differential $\k \frac{dz^2}{|dz|}$ 
 is globally defined on $M$ with values in $\mathcal L V^{\perp} 
 \otimes \mathbb C$ and $\mathcal L = (K \bar K)^{-1/2}$, 
 where $K$ is the canonical bundle \footnote{The real line bundle $\mathcal L = (K \bar K)^{-1/2}$ has been called the 
 \textit{$1$-density bundle}, see \cite{BPP}.}  of $M$ and 
 $\k \frac{dz^2}{|dz|} = \frac{\;{\rm I\!I}^{(2, 0)}}{|dy|}$
 holds, where ${\rm I\!I}^{(2, 0)}$ denotes the $(2, 0)$-part of the normal bundle-valued second fundamental form of $y$. 
\end{Proposition}
\begin{Remark}
 For another complex coordinate $w$, the canonical lift $\widetilde Y$ with respect to 
 $w$ is given by $\widetilde Y = Y \left|w_z\right|$ 
 in terms of the canonical lift $Y$ with respect to coordinate $z$
  because of \eqref{eq:canonical}. Then by
 the invariance of $\k \frac{dz^2}{|dz|}$, the coefficient function 
 $\widetilde \k$ with respect to $w$ can be represented as
\begin{equation}\label{eq:kappainw}
 \widetilde \k = \k \frac{|w_z|}{w_z^{2}}.
\end{equation}
 Moreover, the Schwarzian derivative $S_z (y)$ 
 satisfies the transformation rule: 
\begin{equation}\label{eq:Schwarzian}
 S_w (y) dw^2 = (S_z (y) - S_z (w))dz^2,
\end{equation}
 where $S_z (w)$ is the 
 classical Schwarzian derivative of $w$ with respect to $z$.
 More explicitly, $s = S_z(y)$ and $\tilde s = S_w(y)$ 
 satisfy the following transformation rule, 
\begin{equation}\label{eq:tildes}
 \tilde s =
\left(\frac1{w^{\prime}}\right)^2 \left(s- \left(\frac{w^{\prime \prime}}{w^{\prime}}\right)^{\prime}+\frac12 \left(\frac{w^{\prime \prime}}{w^{\prime}}\right)^{2} \right).
\end{equation}
 Here we use the abbreviation $w^{\prime} = w_z$.
\end{Remark}



\begin{Theorem}[\cite{Bryant1984}, \cite{Ejiri1988}]\label{Pp:Willmore}
 The M\"obius transformation invariant energy 
 \[
 W(y) = 2 \sqrt{-1}\int_M \langle \k, \, \overline \k \rangle\, dz
 \wedge d \bar z
 \]
 will be called the {\rm Willmore energy} of a conformal immersion 
 $y : M \to \mathbb S^n$.
 Moreover, a surface which attains a critical point of the Willmore energy 
 will be called a {\rm Willmore surface}. Then  
 a conformal immersion $y: \mathbb D \to \mathbb S^n$ is \textit{Willmore} 
 if and only if 
\begin{equation*}
D_{\bar z}  D_{\bar z} \k + \frac12 \bar s\k =0 
\tag{Willmore}
\end{equation*}
 holds. Here $\k$ and $s$ are the Hopf differential and the Schwarzian derivative 
 of the surface $y$, respectively.
 Moreover, $y: \mathbb D \to \mathbb S^n$ is \textit{Willmore} if and only if the conformal
 Gauss map $V:\mathbb D \to \Grf$ is conformal harmonic.
\end{Theorem}


\section{Lorentzian $2$-plane lifts ($\LL$-lifts) of a M\"obius surface}
\subsection{Lorentzian $2$-plane lifts ($\LL$-lifts)}
 The conformal Gauss map $\G$ becomes 
 degenerate at umbilic points.  H\'elein \cite{Helein} defined an alternative 
 map which does not have that problem, essentially as follows.
 
For a Willmore surface $y$ with a dual Willmore surface $\hat y$, there is (provided $Y \wedge \widehat Y \neq 0$) a naturally defined bundle of Lorentzian
2-planes defined by $[ Y \wedge \widehat Y]  : M \to \Gr =
 \SOn/(\SOo \times {\rm SO}(n))$.   We generalize this first to an arbitrary M\"obius  surface:
 \begin{Definition}\label{def:llift}
 Let $y: M \to \SSS^n$ be a M\"obius   surface, and let $V$ be as defined above.
A \emph{Lorentzian $2$-plane lift} 
 or \emph{$\LL$-lift} is 
 a map $\Z: M \to \Gr$ such that, for any lift $Y$ of $y$, we have 
 $Y(p) \subset \Z(p) \subset V_p$ for all $p \in M$.
 \end{Definition}
 
\begin{Remark}\label{rmk:global} We will be interested in \emph{conformal} $\LL$-lifts. These exist on a neighbourhood 
of a point (see below), but global existence would
  induce a bundle decomposition $V = Z\oplus Z^{\perp}$ and it may be only true under a topological assumption. However, it is known 
 that global conformal $\LL$-lifts exist for special classes of
 Willmore surfaces:
\blue{ \begin{enumerate}
     \item Every minimal surface in the space forms $\mathbb S^n$,   $\mathbb R^n$ and $\mathbb H^n$ has a natural conformal $\LL$-lift associated to its dual surface \cite{Helein,Helein2, Xia-Shen}.  
 \item  Under some conditions, in the sense of Bryant \cite{Bryant1984} or Babich-Bobenko \cite{Ba-Bo},  a minimal surface  in
 $\mathbb R^n$ or $\mathbb H^n$,
 can extend smoothly across infinity by conformally embedding $\mathbb R^n$ or $\mathbb H^n$ into $\mathbb S^n$. But  the $\LL$-lifts 
associated to the dual surface cannot extend to the possible ends or infinite boundary since these points are the umibilic points of the surfaces. 
 \item  The homogeneous Willmore torus constructed by Ejiri \cite{Ejiri1988} is not S-Willmore and it admits two
 different global adjoint surfaces, thus two different global  conformal $\LL$-lifts. 
 \end{enumerate}}
 \end{Remark}
 If $Y$ is a canonical lift, then we may construct a Lorentzian
$2$-plane lift by considering a section 
$\widehat Y \in \Gamma (V)$ satisfying 
\begin{equation}\label{eq:yhatdef}
\langle \widehat Y, \widehat Y \rangle =0, \quad \langle \widehat Y, Y \rangle =-1.
\end{equation}
  Then a $2$-plane bundle given by 
$\Z= [Y \wedge \widehat Y]$
is well defined:
the choice of $\widehat Y$ does not depend on the conformal coordinate:  if $w=w(z)$ then the 
 canonical lift $Y$ changes to $\tilde Y = |w_z| Y$ and  $\widehat {\tilde Y}$ defined 
by $|w_z|^{-1} \widehat Y$ also satisfies \eqref{eq:yhatdef},  and thus
$Y \wedge \widehat Y = \tilde Y \wedge \widehat {\tilde Y}$ is well defined.
\begin{Lemma} Let $\Z: M \to \Gr$ be a Lorentzian $2$-plane subbundle of $M \times \real^{n+1,1}$, spanned by null vector fields
$Y$ and $\widehat Y$ satisfying $\langle Y, \hy \rangle = -1$.
Then $\Z$ is conformal  if and only if the function $\theta$ defined by$:$
\[
\theta := 2\langle Y \wedge Y_z,  \widehat Y \wedge \widehat Y_z \rangle
\]
is zero.
\end{Lemma}
\begin{proof}
A computation using $\langle Y, Y \rangle = \langle \widehat Y, \widehat Y \rangle = 0$ and $\langle Y, \hy \rangle = -1$.
\end{proof}
\subsection{Existence of conformal $\LL$-lifts and the 
Riccati equation}
We now consider the problem of finding $\widehat Y$ such that $\Z=[Y \wedge \widehat Y]$ is conformal. 
To get an explicit condition on $\hy$, at least locally, we choose a canonical frame for $V$,
and write down the equations that the coefficients of $\hy$ in this basis must satisfy.

For the bundle $V
 = \operatorname{span}\left\{ Y, \, \Re Y_z,\, \Im Y_z,\, Y_{z \bar z}  \right\}$, 
 we define an alternative basis 
 \[\{Y, \, \Re Y_z,\, \Im Y_z,\, N\},\]
 where $N$ is 
 the unique section of $\Gamma (V)$ such that 
\begin{equation}\label{eq:N}
  \langle N, N\rangle = 0, \quad  
 \langle N, Y\rangle = -1, \quad \langle N, Y_z \rangle =  \langle N, Y_{\bar z} \rangle =0.
\end{equation}
 Since $V$ is a rank four sub-bundle, the above equations determine the
 unique section $N$. However, the conditions on $N$ are not invariant under
 coordinate changes.
 
Given $Y$ and $\hat Y$ as in the previous section, with 
\begin{equation}\label{eq:yhat}
\langle \widehat Y, \widehat Y \rangle =0, \quad \langle \widehat Y, Y \rangle =-1,
\end{equation}
 in $\mathbb R^{n+1, 1}$, 
 let us define a complex function $\mu : \mathbb D \to \mathbb C$ as follows:
 \begin{equation}\label{eq:Yhat}
 \widehat Y = N + \bar \mu  Y_z + \mu  Y_{\bar z} 
                + \frac12 |\mu|^2 Y. 
\end{equation}
  Under the  change of coordinate $w= w(z)$ the canonical lift 
 $Y$ changes as  $|w_z| Y$ and $\widehat Y$ changes 
 $|w_z|^{-1} \widehat Y$.
 Thus the complex function 
 $\mu$ has the following 
 transformation rule:
\begin{equation}\label{eq:tildemu}
 \mu = \tilde \mu w_z+ \frac{w_{zz}}{w_z}.
\end{equation}
 We now consider the derivative of $\widehat Y$ 
 from \eqref{eq:Yhat}, and by using \eqref{eq:Hill} we obtain
\[
 \widehat Y_z = \frac{\mu}2 \widehat Y 
 + \theta \left(Y_{\bar z} + \frac{\bar \mu}2 Y\right)
 + \rho \left(Y_{z} + \frac{\mu}2 Y \right)
 +  L,
\]
 where we define
\begin{align}
\label{eq:thetarhozeta1} \theta &= \mu_z - \frac12{\mu^2}-s \in \C,  \\
\label{eq:thetarhozeta2} \rho &= \bar \mu_z - 2 \langle \k, \overline \k \rangle \in \C,  \\
\label{eq:thetarhozeta3} L&=  2 D_{\bar z}\k + {\bar \mu \k} \in \Gamma(V_{\mathbb C}^{\perp}),
\end{align}
 and $D$ denotes the normal connection. Note that the function $s$ and 
 the section $\k$ are respectively the Schwarzian derivative and 
 the Hopf differential. 
  A straightforward computation shows that $\theta$ in 
 \eqref{eq:thetarhozeta1} and
 $\rho$ in \eqref{eq:thetarhozeta2} can be computed as
 \[
 \theta = 2 \la Y \wedge Y_z, \widehat Y \wedge \widehat Y_z \ra, \quad 
\rho = 2 \la Y \wedge Y_{\bar z}, \widehat Y \wedge \widehat Y_z \ra,
 \]
 and thus $\theta dz^2$ and $\rho\, dz d \bar z$ are well-defined 
 $(2,0)$- and $(1, 1)$-forms, respectively.

 It is known that $\theta =0$ can be locally solved for 
 a complex function $\mu$:
 \begin{Lemma}[Riccati equation]\label{lem:Riccati}
 For a given function $s$, there exist local solutions $\mu$ to the following 
 differential equation
 \begin{equation}\label{eq:Riccati}
  \mu_z - \frac12{\mu^2}-s=0.  
 \end{equation}
 The equation \eqref{eq:Riccati} will be commonly called the {\rm Riccati equation}.
 Moreover,  if $\mu$ is a solution of the Riccati equation \eqref{eq:Riccati}, 
 then $\tilde \mu$ in \eqref{eq:tildemu} is a solution of 
 $\tilde \mu_w - \frac12{\tilde \mu^2}-\tilde s=0$,
 where $\tilde s$ is given in \eqref{eq:tildes}.
 \end{Lemma}
 \begin{proof}
 The proof of existence  can be found  in \cite[Lemma 3]{Helein}, and using the transformation rule of the 
 Schwarzian derivative for $\tilde s$ in \eqref{eq:tildes}
 and $\tilde \mu$ in \eqref{eq:tildemu} implies that 
 $\tilde \mu$ is again a solution of the Riccati equation.
 \end{proof}
\begin{Remark}
 Any solution $\mu$ of the Riccati equation \eqref{eq:Riccati}
 is only locally defined in general. But see also Remark \ref{rmk:global}.
\end{Remark}
 A straightforward computation shows that
 for $Z = [Y \wedge \widehat Y]$
\begin{equation}\label{eq:conformal}
 \langle d Z, dZ\rangle = \theta dz^2 + \frac12(\rho + \bar \rho)
(d z d \bar z + d \bar z d z) + \bar \theta d\bar z^2
\end{equation}
holds, where $\theta$, $\rho$ is given in \eqref{eq:thetarhozeta1} and \eqref{eq:thetarhozeta2}. 
 Combining Lemma \ref{lem:Riccati}, we obtain:
\begin{Lemma}\label{lem:conformal}
Let $Z = [Y \wedge \widehat Y]$ be an $\LL$-lift of 
a M\"obius surface $y$. Moreover, let $\mu$ be the 
complex function defined in \eqref{eq:Yhat}.
Then the following statements are equivalent$:$
\begin{enumerate}
\item $\Z$ is conformal.
\item $\mu$ is a solution of the Riccati equation 
\eqref{eq:Riccati}. 
\item  $Y$ and $\widehat Y$ satisfies
 $\widehat Y_z \in \operatorname{span}_\C \{ \widehat Y, Y, Y_z \}$
 modulo $V_{\mathbb C}^{\perp}$,
 where $V_{\mathbb C}^{\perp}$ denotes the complement of 
 the complexification of the $4$-dimensional 
 Lorentzian subspace $V$.
\end{enumerate}
\end{Lemma}
 Note that from \eqref{eq:conformal}, if
 $Z$ is conformal, the conformal factor is $2 \Re \rho$.

\subsection{The moving frame and the structure equations}
 From now on we only consider a solution $\mu$
 of the Riccati equation for a Mobius surface in $\mathbb S^n$, that is, we choose a section $\widehat Y  \in \Gamma (V)$ such that $\la \widehat Y, \widehat Y \ra = 0$ and
 $\la \widehat Y, Y \ra = -1$ hold, and $Z = [Y \wedge \widehat Y]$ is 
 conformal. Then $\widehat Y$ is well-defined, 
 $\mu dz = 2 \langle \widehat Y, Y_z\rangle dz$ becomes a complex connection $1$-form
 on a surface $y$ and $\widehat Y$ defines another surface 
 $\hat y = [\widehat Y]$.
 Note that the surface $\hat y$ can be degenerate in general, and 
 moreover $\hat y$ could coincide with $y$ at some points.

 Let $\{\psi_1, \dots, \psi_{n-2}\}$ be an orthonormal frame of the normal bundle $V^{\perp}$
 and define complex functions  $\gamma_{j}, k_j$ and 
 $d_{ij} (1 \leq i, j \leq n-2)$ by
\begin{align}\label{eq:normalcomponent}
L :=\sum_{j=1}^{n-2} \red{\frac1{2}}\gamma_j \psi_j, \quad 
\k := \sum_{j=1 }^{n-2} \red{\frac12} k_j \psi_j,\quad 
 D_z \psi_i := \sum_{j= 1}^{n-2} \red{\frac12d_{ji}} \psi_j.
\end{align} 
 Note that $L=2  D_{\bar z}\k +  \bar \mu \k$, and 
 since $\{\psi_1, \dots, \psi_{n-2}\}$ is an orthonormal basis, $d_{ij}+ d_{ji}=0 \, (i,j =1, \dots, n-2)$ 
 hold.\footnote{
 The section $L$ is equal to  $2 \zeta$ of \cite[p. 419]{BW}.
 We have chosen factors, respectively, $1/2$ 
 in \eqref{eq:normalcomponent} comparing to \cite[Proposition 7.4]{BW}, 
 since the Maurer-Cartan form in the following can be written in a simple manner.}
 Then the fundamental system of a M\"obius surface in $\mathbb S^n$ can be summarized as follows:
\begin{Proposition}\label{Pp:Movingframe}
 Let $Y$ be the canonical lift of a M\"obius surface $y$
 and 
 $Z = [Y \wedge \widehat Y]$ a conformal $\LL$-lift.
 Moreover, define four sections of $V$ by
 \begin{equation}\label{eq:P12}
 X^{\pm} = \frac1{\sqrt{2}} (\pm Y + \widehat Y), \quad
 P^{+} = \Re (2Y_{\bar z} +  \bar \mu Y), \quad
  P^{-} = \Im (2Y_{\bar z} +  \bar \mu Y), 
\end{equation}
 and  denote by $\{\psi_1, \dots, \psi_{n-2}\}$ the orthonormal frame of the normal bundle $V^{\perp}$
 in \eqref{eq:normalcomponent}.
 Set 
\begin{equation}\label{eq:adaptedframe}
 F = 
\left(
X^+, \,
 X^-,\, P^+, P^-, 
 \psi_1, \dots, \psi_{n-2}
\right)
: \mathbb D \to \SOn.
\end{equation}
 and take $\alpha = F^{-1} d F $ the Maurer-Cartan form of $F$ and 
 decompose it by $(1,0)$- and $(0, 1)$- parts as $\alpha = \alpha^{\prime} + \alpha^{\prime \prime}$
 $(\alpha^{\prime \prime} = \overline{\alpha^{\prime}})$ and set
\begin{equation}\label{eq:alphaprime}
 \alpha^{\prime} 
 = \begin{pmatrix}
 A & B^{\sharp} \\
 B   & C
 \end{pmatrix} dz, \quad B^{\sharp}= \di (1, -1)B^T,
\end{equation}
 where $A$, $B$ and $C$ respectively take values in  $\mathfrak {so}(1,1, \mathbb C),  
 M(n, 2, \mathbb C)$ and $\mathfrak{so}(n, \mathbb C)$.
 Then the matrix valued functions $A, B$ and $C$ can be explicitly computed by
 the complex functions  $\rho$, $\mu$, $\gamma_j, k_j (j=1, \dots, n-2)$ and 
 $d_{ij} (i, j=1, \dots, n-2)$ in \eqref{eq:thetarhozeta2},
 \eqref{eq:Riccati}, \eqref{eq:normalcomponent}$:$
\begin{align}\label{eq:A1}
 A &= \frac12 \begin{pmatrix}
 0 & \mu\\ \mu & 0
 \end{pmatrix} ,\\
\label{eq:B1}
 B &= 
\frac1{2 \sqrt{2}}
  \begin{pmatrix}
 \rho^+  &  \rho^-  \\
 \gamma& \gamma 
\end{pmatrix}, \quad \rho^{\pm} =  (\rho \pm 1)\begin{pmatrix} 1  \\ -\sqrt{-1}\end{pmatrix}, \quad 
\gamma = \begin{pmatrix} \gamma_1 \\ \vdots \\ \gamma_{n-2}\end{pmatrix},\\
 \label{eq:C}
 C &=  \frac12 
  \begin{pmatrix}
 0 & \sqrt{-1} \mu & - k^T \\
 -\sqrt{-1} \mu & 0 & - \sqrt{-1}k^T \\
 k  &  \sqrt{-1}k& d
\end{pmatrix}, \,
k = 
\begin{pmatrix}
k_1 \\
\vdots\\
k_{n-2}    
\end{pmatrix}, \,
d= (d_{ij}) \in \mathfrak{so}(n-2, \C).
\end{align}
 Moreover, the compatibility condition $F_{z \bar z} = F_{\bar z z}$ is equivalent to 
 the conformal Gauss, Codazzi and Ricci equations for a M\"obius surface $y: \mathbb D
 \to \mathbb S^n:$
\begin{align*}
& \frac12 s_{\bar z} = 3 \langle \k, D_z \overline \k \rangle  + \langle D_z \k, \overline \k \rangle, 
\tag{Gauss } \\
&\Im (D_{\bar z} D_{\bar z} \k + \frac12\bar s\k)=0, \tag{Codazzi}\\
&D_{\bar z} D_{z} \psi_j - D_{z} D_{\bar z} \psi_j = 
2 \langle \psi_j, \k \rangle \overline \k -2 \langle \psi_j, \overline \k \rangle \k, \quad (j=1, \dots, n-2). \tag{Ricci}
\end{align*}
 Note that $R^D_{\bar z z} \psi_j= D_{\bar z} D_{z} \psi_j - D_{z} D_{\bar z} \psi_j$ holds and
 $s$ and $\k$ are the Schwarzian derivative and the Hopf differential, respectively.
 The orthonormal frame $F$ in \eqref{eq:adaptedframe} will be called the 
 {\rm adapted frame}. 
\end{Proposition}
\begin{Remark}
 Note that the sections $P^{\pm}$ in \eqref{eq:P12} can also be represented as
\begin{equation}\label{eq:Ppm}
 P^{+} = \Re (2Y_{z} +  \mu Y), \quad P^{-} = -\Im (2Y_{z} +  \mu Y).
\end{equation}

\end{Remark}
\begin{proof}For the canonical lift $Y$, the unique section $N$ in \eqref{eq:N} and any section $L \in \Gamma (V_{\C}^{\perp})$,
 the following equations hold$:$
\begin{equation}\label{eq:fundamental}
\left\{
\begin{array}{ll}
Y_{z z } &= - \frac{s}2 Y + \k, \\   
Y_{z \bar z }&= - \la \k, \overline \k \ra Y + \frac12 N, \\
N_{z}&= - 2 \la \k, \overline{\k} \ra Y_z - s Y_{\bar z} + 2 D_{\bar z} \k, \\   
L_{z }&= D_z L + 2 \la L, D_{\bar z} \k \ra Y -2 \la L, \k \ra Y_{\bar z}, 
\end{array}
\right.
\end{equation}
 where $s$ and $\k$ are respectively the Schwarzian derivative and the Hopf differential.
 Note that the first equation is nothing but the inhomogeneous Hill's equation in \eqref{eq:Hill}, and 
 the other three equations follow from a straightforward computation. 
 Consider the new framing  $\{Y,\widehat Y, P^{\pm}\}$ of $V$, by \eqref{eq:P12}.
 We reformulate \eqref{eq:fundamental} as 
 (Here we set $P:=P^+-\sqrt{-1}P^-$)
\begin{equation*}
\left\{
\begin{array}{ll}
Y_{z} &=  - \frac{\mu}{2} Y +\frac{1}{2} P,\\[0.1cm]
\hy_{z}&=  \frac{\mu}{2}\hy + \frac{\rho}{2} P +L, \\ 
 P_{z } &=  \frac{\mu}{2}P  + 2\k, \\   
P_{\bar{z}}&=-\frac{\bar{\mu}}{2}P+\bar{\rho}Y+\widehat{Y}, \\
L_{z }&= D_z L +  \la L, L\ra Y -\la L, \k \ra \bar{P}. 
\end{array}
\right.
\end{equation*}
So we obtain 
 \begin{equation*}
\left\{
\begin{array}{ll}
X^{\pm}_{z} &=  \frac{\mu}{2} X^{\mp} +\frac{\pm1+\rho}{2\sqrt{2}} (P^+-\sqrt{-1}P^-)+\red{\frac1{\sqrt{2}}}L,\\
P^{+}_{z } &=  \frac{-\sqrt{-1}\mu}{2}P^- +\frac{1+\rho}{2\sqrt{2}} X^+-\frac{-1+\rho}{2\sqrt{2}} X^-+\k, \\
P^{-}_{z } &=  \frac{\sqrt{-1}\mu}{2}P^+ -\frac{\sqrt{-1}(1+\rho)}{2\sqrt{2}} X^++\frac{\sqrt{-1}(-1+\rho)}{2\sqrt{2}} X^- +\sqrt{-1}\k,\\
L_{z }&= D_z L + \la L, L\ra Y -\la L, \k \ra (P^++\sqrt{-1}P^-). 
\end{array}
\right.
\end{equation*}
Then the equations \eqref{eq:A1}, \eqref{eq:B1} and \eqref{eq:C} follow from this, i.e., the Gauss equation, Codazzi equation and Ricci equations have been derived in \cite{BPP}.
\end{proof}

\begin{Remark}
 The adapted frame $F$ depends on the choice of a solution $\mu$ of the Riccati equation, 
 and it is locally defined.
 The compatibility $F_{z \bar z} = F_{\bar z z}$ is equivalent to $d \alpha + \frac12 [\alpha\wedge \alpha]=0$,
 where $\alpha = F^{-1} d F $ is the left Maurer-Cartan form of $F$.
 Moreover, it can be interpreted as a flatness of the connection 
 $d + \alpha$.
\end{Remark}
\begin{Corollary}\label{coro:Moving}
The identity 
 $B^T B =
 \red{\frac12 \langle L, L \rangle} \begin{pmatrix} 1 & 1 \\ 1 & 1 \end{pmatrix}$
  holds, where $L$ is the section in \eqref{eq:thetarhozeta3}.
  Moreover, by setting $B = (b_1\, b_2)$, 
 $b_1 - b_2 \neq 0$ holds.
\end{Corollary}
\section{Willmore surfaces and conformal harmonic $\LL$-lifts}\label{sc:Harmoniclift}

\subsection{Willmore surfaces, $\LL$-lifts and  Ruh-Vilms type theorem}
 As we have seen that the connection $1$-form $d + \alpha$ with $\alpha = F^{-1} d F$ is 
 a fundamental object for a M\"obius surface $y$, and we now introduce a family of connections $1$-forms $d + \alpha^{\lambda}$ 
 parameterized by $\lambda \in S^1$ as follows:
\begin{equation}\label{eq:familyflat}
 \alpha^{\lambda} = \lambda^{-1} \alpha^{\prime}_{\mathfrak p} + 
 \alpha_{\mathfrak k} +
\lambda \alpha^{\prime \prime}_{\mathfrak p},
\end{equation}
 where $\mathfrak g = \operatorname{Lie} (\SOn)$ will be decomposed as
 $\mathfrak g= \mathfrak k \oplus \mathfrak p$
 with $\mathfrak k =  \operatorname{Lie}(\SOo \times {\rm SO}(n))$,
 and $\alpha_{\mathfrak k}$ and $\alpha_{\mathfrak p}$ are 
 the $\mathfrak k$- and the $\mathfrak p$-valued $1$-forms.
 Moreover $\prime$ and $\prime \prime$ denote the $(1,0)$- and 
 the $(0,1)$-parts, respectively.
 Note that $\alpha^{\lambda}|_{\lambda =1} = \alpha$, and $d + \alpha^{\lambda}$
 is flat for $\lambda=1$. Then the flatness of a whole family  $d + \alpha^{\lambda}$ 
 parameterized by $\lambda \in S^1$
 gives an additional constraint for a M\"obius surface $y$. 
  We now characterize Willmore surfaces in terms of a conformal $\LL$-lift and the family of connections $d + \alpha^{\lambda}$.
\begin{Theorem}[Ruh-Vilms type theorem]\label{thm:RuhVilms}
 Let $Y$ be the canonical lift of a conformal immersion $y :\D \to \mathbb S^n$ and $Z = [Y \wedge \widehat Y]$ 
  a conformal $\LL$-lift.
 Then the following statements are equivalent$:$
\begin{enumerate}
 \item The surface is a Willmore surface.
 \item The section $L$
  in \eqref{eq:thetarhozeta3} satisfies $2 D_{\bar z} L- \bar \mu L=0$.
 \item The local $\LL$-lift $\Z$  is conformal harmonic.
 \item The family of connections $d + \alpha^{\lambda}$ in \eqref{eq:familyflat} 
 is flat for all $\lambda \in S^1$.
\end{enumerate}
 \end{Theorem}
\begin{proof}
 The equivalence of (3) and (4) is a standard formulation of 
 a harmonic map from a Riemann surface into a symmetric space, since 
 the decomposition of  $\mathfrak g = \operatorname{Lie} (\SOn)$, 
 by $\mathfrak g= \mathfrak k \oplus \mathfrak p$
 with $\mathfrak k =  \operatorname{Lie}(\SOo \times {\rm SO}(n))$,
 is associated to the symmetric space $\Gr$, see for example \cite{DPW}.

 (1) $\Leftrightarrow$ (2): By definition and the Riccati equation of $\mu$, we have 
\begin{align*}
2 D_{\bar z} L-\bar \mu L&=2(2D_{\bar z}D_{\bar z} \k+\bar{\mu}_{\bar{z}}\k+\bar{\mu}D_{\bar z} \k)-\bar \mu(2D_{\bar z} \k+\bar{\mu}\k)=2(2D_{\bar z}D_{\bar z} \k+\bar{s}\k).
\end{align*}
By the Willmore equation, we obtain the conclusion.

 (2) $\Leftrightarrow$ (4): 
We compute the flatness conditions
 $d \alpha^{\lambda} + \frac12 [\alpha^{\lambda} \wedge \alpha^{\lambda}]=0$. 
 It is equivalent to the Gauss, Codazzi and Ricci equations in Proposition \ref{Pp:Movingframe}, 
 together with the extra condition $B_{\bar{z}}+\bar{C}B-B\bar{A} =0$,
 which is equivalent to 
 \begin{align*}
 \rho_{\bar{z}}-\bar{\mu}\rho-2\langle L,\bar{\k}\rangle=0,
 \quad 
\gamma_{i\bar{z}}+\frac12\sum_{j}\gamma_j\overline{\red{d_{ij}}}-\frac{\bar{\mu}}{2}\gamma_i=0, \quad(1\leq i\leq n-2).
\end{align*}
The second equation is equivalent to $2 D_{\bar z} L- \bar \mu L=0$. 
Since  $\rho_{\bar{z}}-\bar{\mu}\rho-2\langle L,\bar{\k}\rangle=0$ holds for all Willmore surfaces ((4.10) of \cite{Ma2006}), we
complete the proof. \end{proof}


 From (4) in Theorem \ref{thm:RuhVilms}, there exists a family of frames $F^{\lambda}$
 from $\mathbb D$ to the following twisted loop group
\begin{equation}\label{eq:loopgroup}
 \Lambda \SOn_{\sigma}= \left\{g: S^1 \to \SOn\mid g(-\lambda) = \sigma g(\lambda)\right\},
\end{equation}
 where $\sigma$ is the involution of the symmetric space $\Gr$ given by 
\begin{equation}\label{eq:sigma}
 \sigma g = \operatorname{Ad} \operatorname{diag}(-1,-1,1,\dots,1) g, \quad g \in \SOn. 
\end{equation}
\begin{Definition}
 The family of moving frames 
 $F^{\lambda} : \mathbb D \to  \Lambda \SOn_{\sigma}$ given by $\alpha^{\lambda} = (F^{\lambda})^{-1} d F^{\lambda}$ of (4) in Theorem \ref{thm:RuhVilms} will be called the \textit{extended frame} of a Willmore surface $y$.
\end{Definition}
 Note that the extended frame $F^{\lambda}$ evaluated $\lambda =1$ 
 is the adapted frame of $y$ up to a left multiplication of a constant factor in $\SOn$.
 
\subsection{Adjoint transformation, S-Willmore surfaces and  $\LL$-lifts}\label{sbsc:SWill}
 Let us recall the Willmore condition
 $D_{\bar z}  D_{\bar z} \k + \frac12 \bar s\k =0$. 
 By the Codazzi equation, the Willmore condition is also equivalent to 
 $\Re(D_{\bar z}  D_{\bar z} \k + \frac12 \bar s\k) =0$.
 Then we define the special class of Willmore surfaces, \cite[Theorem 1.1]{Ejiri1988}.
\begin{Definition}\label{def:SWill}
 A Willmore surface $y : M \to \mathbb S^n$ is called an \textit{$S$-Willmore surface} if $D_{\bar z} \k \parallel \k$ holds, that is, there exists some function $r$ such that $D_{\bar z} \k + r \k=0$ holds.
\end{Definition}
 As seen in Proposition \ref{Pp:Willmore},
 the Willmore condition for a M\"obius surface with a Riccati equation 
 $\mu_z - \frac12{\mu^2}-s =0$ is characterized as
\begin{equation}\label{eq:Willmore}
  D_{\bar z} L -\frac{\bar \mu}2L=0,
\end{equation}
 where $L=  2 D_{\bar z} \k + \bar \mu \k$.
 Then this condition means that the covariant derivative of $L$
 is parallel to itself $D_{\bar z} L \parallel L$, and 
 thus we could consider the following sufficient conditions:
 $\langle L, L\rangle =0$ or $L=0$.
 This observation gives the following results.
\begin{Lemma}\label{lem:S-A}
 Let $y$ be a M\"obius surface in $\mathbb S^{n}$ with a 
 conformal $\LL$-lift $Z = [Y \wedge \widehat Y]$.
 Then the following statements hold$:$
\begin{enumerate}
\item  The surface $y$ is \textit{S-Willmore} with a dual surface $\hat{y}=[\hat{Y}]$ if and only if 
$L=0$.

 \item If $y$ is a Willmore surface and $\la L, L\rangle=0$ is satisfied, then the surface $\hat y =[\widehat Y]$ with $\widehat Y$ in \eqref{eq:Yhat} 
 is a Willmore surface, and it will be called the {\rm adjoint transformation} of $y$. 
\end{enumerate}
\end{Lemma}
\begin{proof}
 The statement (1) is a straightforward computation by using Definition \ref{def:SWill}.
 The statement (2) has been proven in \cite[Section 4.2]{Ma2006}. 
\end{proof}  
\begin{Remark}\label{rm:adjointS}
\mbox{}
\begin{enumerate}
 \item The adjoint transformation of a Willmore surface 
 was defined in \cite[Definition 4.2]{Ma2006}.
 From Theorem 4.8 in \cite{Ma2006}, there always 
 exist adjoint transformations of a Willmore surface $y$ and 
 then $\hat y = [\widehat Y]$ is also a Willmore surface if it is not 
 degenerate, and $\hat y$ will be called an \textit{adjoint surface}. Moreover adjoint transformations of a Willmore surface $y$ 
 are not unique, since solutions of $\theta =0$ and
 $\langle L, L\rangle=0$ are not unique, in general.

\item  When $n=3$, any Willmore surface is an S-Willmore surface, 
 since $L$ is a complex function and $\langle L, L\rangle=0$ is
 equivalent to $L=0$ and the existence of adjoint transformations as 
 given in the above $(2)$.
 For $n> 4$, there exist Willmore surfaces which are not S-Willmore,
 \cite{Ejiri1982}.

\end{enumerate}
\end{Remark}
 We now consider singularities of the $\LL$-lift $Z = [Y \wedge \widehat Y]$
 of a Willmore surface $Y$ and its adjoint transformation $\widehat Y$.
 It may occur on umbilic points. 
 \begin{Proposition}\label{pp:umbilic} \cite{Ma,Ma2006}
 Let $y$ be a Willmore surface on $M$ and $M_0 \subset M$ be the umbilic free 
 subset of $y$ on $M$. Then for any $z_0 \in M_0$,
 there exists an open set $U \subset M_0$ containing $z_0$ such that 
 the canonical lifts of $y$ and the adjoint transformation $\hat y$, which 
 will be denoted by $Y$ and $\widehat Y$, define the conformal harmonic 
 $\LL$-lift $Z = [Y \wedge 
 \widehat Y]$ on $U$. 
\end{Proposition}
On the other hand, 
 for an umbilic point $z_1 \in M$ and an open set $U_1 \subset M$
 containing $z_1$, the $\LL$-lift $Z = [Y \wedge \widehat Y]$
 may not be well-defined on $U_1$ in some choice of $\mu$.
 Therefore, we need to think about the original Willmore conditions
 \eqref{eq:Willmore} in general to have a well-defined conformal 
 $\LL$-lift.

 Recall that the $n\times 2$ matrix $B$ in the Maurer-Cartan form is given in \eqref{eq:B1} and its identity in Corollary \ref{coro:Moving}.
From Lemma \ref{lem:S-A} we obtain:
\begin{Proposition}\label{pp:various}
\begin{enumerate}
\item Let $y$ be a Willmore surface and $\hat y$ be an adjoint transformation 
 of $y$ such that $Z = [Y \wedge \widehat Y]$ is well-defined. Then 
 the $n \times 2$ matrix $B$ in the  Maurer-Cartan form of the adapted frame in 
 \eqref{eq:B1} satisfies 
\begin{equation}\label{eq:adjointharmonic}
 B^T B = 0.
\end{equation}
 In this case, the conformal harmonic $\LL$-lift $\Z$ will be called 
  {\rm isotropic}.

\item 
 Let $y$ be an S-Willmore surface and $\hat y$ a dual surface 
 such that $Z = [Y \wedge \widehat Y]$ is well-defined. Then 
 the $n \times 2$ matrix $B$ in the  Maurer-Cartan form of the adapted frame in 
 \eqref{eq:B1} satisfies 
\begin{equation}\label{eq:SharmonicGauss}
 B^T B = 0, \quad \operatorname{rank} B =1.
\end{equation}
 In this case, the conformal harmonic $\LL$-lift $\Z$ will be called 
 the {\rm S-isotropic}.

\end{enumerate}
\end{Proposition}
The results are summarized as the following table$:$
\begin{center}
\begin{tabular}{c|c|c}

 Surface classes & Conformal $\LL$-lifts & $n \times 2$ matrix $B$\\
  \hline &&\\[-0.2cm]

 Willmore on $M$  & Conformal harmonic  & $ B^T B = { \frac12 \la L, L\ra} \left(\begin{matrix}
     1& 1\\ 1 & 1
 \end{matrix}\right) $ \\[0.3cm]
 
 \hline &&\\[-0.2cm]
 \red{
 \begin{tabular}{c}
Willmore with adjoint \\ 
 transformation on $M_0$ 
\end{tabular}
}
 & Isotropic conformal harmonic & $ B^T B = O_2$ \\[0.2cm]
  \hline &&\\[-0.2cm]
 
 S-Willmore on $M_0$ & S-isotropic conformal harmonic & $ B^T B = O_2$  and $\operatorname{rank} B = 1$\\

\end{tabular}
\end{center}
 Since $\mu$ is not a unique solution of the Riccati equation, thus the conformal harmonic $\LL$-lift $\Z$ of a surface (S-) Willmore is not necessarily (S-) isotropic, that is, $\mu$ may not be a solution of $\la L, L\ra=0$ (or $L=0$) 
 but a solution of $2 D_{\bar z} L- \bar \mu L=0$.
 In fact, a conformal harmonic (S-) isotropic $\LL$-lift $\Z= [Y \wedge \widehat Y]$ may not be well-defined around an umbilic point by Proposition \ref{pp:umbilic}.
 
{\bf Open questions}:
We would like to ask some open problems that originally come from H\'elein's work \red{\cite{Helein,Helein2}}.
\begin{enumerate}
    \item For a closed, oriented, Willmore surface $y$ in $S^n$, when will it admit a global isotropic conformal harmonic $\LL$-lift?
    \item We might split this problem into two sub-questions:
\begin{enumerate}
    \item If $y$ is S-Willmore, does it always admit a global  isotropic  conformal harmonic $\LL$-lift?    
    \item If $y$ is not S-Willmore, does it always admit a global isotropic conformal harmonic  $\LL$-lift?
\end{enumerate}
\end{enumerate}

\bibliographystyle{plain}
\bibliography{mybib}
\end{document}